\documentclass[a4paper,10pt]{amsart}
\usepackage[utf8x]{inputenc}

\usepackage{textcomp}

\title[Local existence for non-resistive MHD in Sobolev spaces]{Local existence for the non-resistive MHD equations in nearly optimal Sobolev spaces}

\author[C.\ L.\ Fefferman]{Charles L.\ Fefferman}
\thanks{JLR was supported by the European Research Council (grant agreement \textnumero\ 616797 - Rodrigo).}
\address{C.\ L.\ Fefferman \\
Department of Mathematics \\
Princeton University \\
Fine Hall \\
Washington Road \\
Princeton, NJ 08544}
\email{cf@math.princeton.edu}

\author[D.\ S.\ McCormick]{David S.\ McCormick}
\address{D.\ S.\ McCormick \\
School of Mathematical and Physical Sciences \\
Pevensey 2 \\
University of Sussex \\
Brighton, BN1 9QH \\
United Kingdom}
\email{d.s.mccormick@sussex.ac.uk}

\author[J.\ C.\ Robinson]{James C.\ Robinson}
\address{J.\ C.\ Robinson \\
Mathematics Institute \\
University of Warwick \\
Coventry, CV4 7AL \\
United Kingdom}
\email{j.c.robinson@warwick.ac.uk}

\author[J.\ L.\ Rodrigo]{Jose L.\ Rodrigo}
\address{J.\ L.\ Rodrigo \\
Mathematics Institute \\
University of Warwick \\
Coventry, CV4 7AL \\
United Kingdom}
\email{j.l.rodrigo@warwick.ac.uk} 

\date{February 8, 2016}

\usepackage{amsmath,amssymb,amsthm,amsxtra,mathrsfs,bm,bbm}
\usepackage{graphicx}
\usepackage{natbib}
\usepackage{url}

\def\e{{\rm e}}
\def\d{{\rm d}}
\def\<{{\langle}}
\def\>{{\rangle}}
\def\R{{\mathbb R}}

\def\eps{{\varepsilon}}

\makeatletter
\newcommand{\@citebnostar}[1]{(\citeauthor{#1}, \citeyear{#1})}
\newcommand{\@citebstar}[1]{(\citeauthor*{#1}, \citeyear{#1})}
\def\citeb{\@ifstar\@citebstar\@citebnostar}
\makeatother

\newtheorem{theorem}{Theorem}[section]

\newtheorem{corollary}[theorem]{Corollary}
\newtheorem{lemma}[theorem]{Lemma}
\newtheorem{proposition}[theorem]{Proposition}

\numberwithin{equation}{section}

\begin{document}

\begin{abstract}
This paper establishes the local-in-time existence and uniqueness of solutions to the viscous, non-resistive magnetohydrodynamics (MHD) equations in $\R^d$, $d=2,3$, with initial data $B_0\in H^s(\R^d)$ and $u_0\in H^{s-1+\eps}(\R^d)$ for $s>d/2$ and any $0<\eps<1$. The proof relies on maximal regularity estimates for the Stokes equation. The obstruction to taking $\eps=0$ is explained by the failure of solutions of the heat equation with initial data $u_0\in H^{s-1}$ to satisfy $u\in L^1(0,T;H^{s+1})$; we provide an explicit example of this phenomenon.
\end{abstract}

\maketitle

\section{Introduction}

In this paper we consider the equations of MHD with zero magnetic resistivity,
\begin{subequations}
\label{MHD}
\begin{align}
  u_t-\Delta u+(u\cdot\nabla)u+\nabla p&=(B\cdot\nabla)B,\qquad\nabla\cdot u=0,\\
  B_t+(u\cdot\nabla)B&=(B\cdot\nabla)u,\qquad\nabla\cdot B=0,
\end{align}
\end{subequations}
along with specified initial data $u(0)=u_0$ and $B(0)=B_0$. \cite{art:JiuNiu2006} established the local existence of solutions in 2D for initial data in $H^s$ for integer $s\ge 3$, and in a previous paper \citeb{art:Commutators} we proved a local existence result for these equations when $u_0,B_0\in H^s(\R^d)$ with $s>d/2$ for $d=2,3$. However, given the presence of the diffusive term in the equation for $u$, it is natural to expect that such a local existence result should be possible with less regularity for $u_0$ than for $B_0$.

To this end, it was shown in \cite{art:BesovMHD} that one can prove local existence when $B_0\in B^{d/2}_{2,1}(\R^d)$ and $u_0\in B^{d/2-1}_{2,1}(\R^d)$. The underlying observation that allowed for such a result is that when $u_0\in B^{d/2-1}_{2,1}(\R^d)$ the solution $u$ of the heat equation with initial data $u_0$ is an element of $L^1(0,T;B^{d/2+1}_{2,1}(\R^d))$: in these Besov spaces the solution regularises sufficiently that an additional two derivatives become integrable in time.

In Sobolev spaces this does not occur: in Lemma \ref{lem:good} we show that for $u_0\in H^s$ we have
$$
u\in L^\infty(0,T;H^s)\cap L^2(0,T;H^{s+1})\cap L^q(0,T;H^{s+2})
$$
for any $0<q<1$, and this is in some sense the best possible. The failure of the estimate $u\in L^1(0,T;H^{s+2})$ is `well known', but it is not easy to find any explicit example in the literature, so we provide one here in Lemma~\ref{lem:bad}.

In this paper we therefore take $B_0\in H^s(\R^d)$, with $s>d/2$, and $u_0$ slightly more regular than $H^{s-1}(\R^d)$, namely $u_0\in H^{s-1+\eps}(\R^d)$, with $0<\eps<1$. By making use of maximal regularity results for the heat equation (which we recall in the next section) we are then able to prove the local existence of a solution that remains bounded in these spaces (see Theorem~\ref{NLE} for a precise statement).

Throughout the paper we use the notation $\Lambda^s$ to denote the fractional derivative of order $s$, given in terms of the Fourier transform by
$$
\widehat{\Lambda^sf}=|\xi|^s\hat f.
$$
We write
$$
\|u\|_{H^s}^2=\|\Lambda^su\|^2+\|u\|^2,\qquad s>0,
$$
which is equivalent to the standard $H^s$ norm when $s$ is a positive integer.

\section{Estimates for solutions of the heat equation in Sobolev spaces}

\subsection{Energy estimates}

First we prove some standard estimates for solutions of the heat equation, including the $L^q(0,T;H^{s+2})$ estimate for $0<q<1$. We give the proofs, since we will need to keep careful track of the dependence of the estimates on $T$; for simplicity we restrict to $T\le 1$, which will of course be sufficient for local existence arguments.

\begin{lemma}\label{lem:good}
  If $u_0\in H^s(\R^d)$ and $u$ denotes the solution of the heat equation
  $$
  \partial_tu=\Delta u\qquad\mbox{with}\qquad u(0)=u_0
  $$
  then $u\in L^\infty(0,T;H^s)\cap L^2(0,T;H^{s+1})$ and $\sqrt t u\in L^2(0,T;H^{s+2}(\R^d))$. When $T\le 1$
  $$
  \sup_{0\le t\le T}\|u(t)\|_{H^s}^2,\quad\int_0^T\|u(t)\|_{H^{s+1}}^2,\quad\int_0^T t\|u(t)\|_{H^{s+2}}^2\le\|u_0\|_{H^s}^2.
  $$
  Consequently $u\in L^q(0,T;H^{s+2}(\R^d))$ for any $0<q<1$, and
  \begin{equation}\label{twomore}
  \int_0^T\|u(t)\|_{H^{s+2}}^q\,\d s\le C_qT^{1-q}\|u_0\|_{H^s}^q
  \end{equation}
  provided that $T\le 1$.
\end{lemma}

\begin{proof}
We start with the $L^2$ estimate obtained by taking the inner product with $u$ in $L^2$ to give
 $$
 \frac{1}{2}\frac{\d}{\d t}\|u\|^2+\|\nabla u\|^2=0
 $$
 and then integrating in time,
$$
\frac{1}{2}\|u(t)\|^2+\int_0^t\|\nabla u(\tau)\|^2\,\d\tau=\frac{1}{2}\|u_0\|^2.
$$
To bound the higher derivatives we act on the equations with $\Lambda^s$ and then take the $L^2$ inner product with $\Lambda^s u$ to obtain
$$
\frac{1}{2}\frac{\d}{\d t}\|\Lambda^s u\|^2+\|\Lambda^{s+1}u\|^2=0,
$$
followed by an integration in time,
\begin{equation}\label{straightE}
\frac{1}{2}\|\Lambda^su(t)\|^2+\int_0^t\|\Lambda^{s+1}u(\tau)\|^2\,\d\tau=\frac{1}{2}\|\Lambda^su_0\|^2.
\end{equation}
Combining these two estimates shows that $u\in L^\infty(0,T;H^s)\cap L^2(0,T;H^{s+1})$, with
$$
\sup_{0\le t\le T}\|u(t)\|_{H^s}^2\le\|u_0\|_{H^s}^2
$$
and
$$
\int_0^T\|u\|_{H^{s+1}}^2\le\frac{1}{2}\|\Lambda^su_0\|^2+T\|u_0\|^2;
$$
in particular if $T\le 1$ then
$$
\int_0^T\|u\|_{H^{s+1}}^2\le\|u_0\|_{H^s}^2.
$$

To obtain the bound on $\sqrt t u$ in $H^{s+2}$ we act on the equations with $\Lambda^{s+1}$ and take the $L^2$ inner product with $t\Lambda^{s+1} u$. Then
$$
\frac{1}{2}\frac{\d}{\d t}\left(t\|\Lambda^{s+1}u\|^2\right)-\frac{1}{2}\|\Lambda^{s+1}u\|^2+t\|\Lambda^{s+2}u\|^2=0.
$$
Integrating from $0$ to $T$ yields
\begin{equation}\label{tLambda}
\frac{T}{2}\|\Lambda^{s+1}u(T)\|^2+\int_0^T t\|\Lambda^{s+2}u(t)\|^2\,\d t\le\frac{1}{2}\int_0^T\|\Lambda^{s+1}u(t)\|^2\,\d t\le\frac{1}{4}\|\Lambda^su_0\|^2,
\end{equation}
using the bound from (\ref{straightE}). We also have
$$
\int_0^T t\|u(t)\|^2\,\d t\le \frac{T^2}{2}\|u_0\|^2,
$$
from which it follows that $\sqrt t u\in L^2(0,T;H^{s+2})$. For $T\le 1$ we can combine this with (\ref{tLambda}) to obtain the estimate
$$
\int_0^T t\|u\|_{H^{s+2}}^2\,\d t\le \frac{1}{2}\|u_0\|_{H^s}^2.
$$


For any $q<1$ we have, using H\"older's inequality with exponents $2/q$ and $2/(2-q)$,
\begin{align*}
\int_0^T \|u(t)\|_{H^{s+2}}^q&=\int_0^T t^{-q/2}t^{q/2}\|u(t)\|_{H^{s+2}}^q\\
&\le\left(\int_0^T t^{-q/(2-q)}\right)^{1-(q/2)}\left(\int_0^T t\|u(t)\|_{H^{s+2}}^2\right)^{q/2}\\
&\le C_qT^{1-q}\|u_0\|_{H^s}^q,
\end{align*}
since $0<q<1$ ensures that $q/(2-q)<1$ and the first term is integrable.\end{proof}

\subsection{An example of $u_0\in L^2$ with $u\notin L^1(0,T;H^2)$}

We now provide an explicit example of an initial condition $u_0\in L^2(\R^d)$ such that $u\notin L^1(0,T;H^2(\R^d))$.

\begin{lemma}\label{lem:bad}
There exists $u_0\in L^2(\R^d)$ such that the solution $u$ of the heat equation with initial data $u_0$ is not an element of $L^1(0,T;H^2(\R^d))$.
\end{lemma}

\begin{proof}
We let $u_0$ be the function with Fourier transform
\begin{equation}\label{hatuxi}
\hat u_0(\xi)=\frac{1}{|\xi|^{d/2}\log(2+|\xi|)}.
\end{equation}
The solution $u(t)$ of the heat equation with initial data $u_0$ has Fourier transform
$$
\hat u(t,\xi)=\hat u_0(\xi)\e^{-|\xi|^2t}
$$
and
$$
\|u(t)\|_{\dot H^2}^2=\int_{\R^d}|\xi|^4|\hat u_0(\xi)|^2\e^{-2|\xi|^2t}\,\d\xi.
$$
We therefore have
$$
I:=\int_0^T \|u(t)\|_{\dot H^2}\,\d t=\int_0^T\left(\int_{\R^d}|\xi|^4|\hat u_0(\xi)|^2\e^{-2|\xi|^2t}\,\d\xi\right)^{1/2}\,\d t.
$$

In order to bound this from below we split the range of time integration, choosing $j_0$ such that $j_0^{-2}\le T$, and write
\begin{align*}
I&\ge\sum_{j=j_0}^N\int_{(j+1)^{-2}}^{j^{-2}} \left(\int_{\R^d}|\xi|^4|\hat u_0(\xi)|^2\e^{-2|\xi|^2t}\,\d\xi\right)^{1/2}\,\d t\\
&\ge\sum_{j=j_0}^N\int_{(j+1)^{-2}}^{j^{-2}} \left(\int_{|\xi|\le j}|\xi|^4|\hat u_0(\xi)|^2\e^{-2|\xi|^2t}\,\d\xi\right)^{1/2}\,\d t\\
&\ge \e^{-1}\sum_{j=j_0}^N\int_{(j+1)^{-2}}^{j^{-2}} \left(\int_{|\xi|\le j}|\xi|^4|\hat u_0(\xi)|^2\,\d\xi\right)^{1/2}\,\d t\\
&\ge\e^{-1}\sum_{j=j_0}^N \frac{1}{j^2(j+1)}\left(\int_{|\xi|\le j}|\xi|^4|\hat u_0(\xi)|^2\,\d\xi\right)^{1/2}.
\end{align*}

By our choice of $\hat u_0(\xi)$ in (\ref{hatuxi}) we have
\begin{align*}
\int_{|\xi|\le j}|\xi|^4|\hat u_0(\xi)|^2\,\d\xi &=\int_{|\xi|\le j}|\xi|^4\frac{1}{|\xi|^{d}[\log(2+|\xi|)]^2}\,\d\xi\\
&\ge\frac{1}{[\log(2+j)]^2}\,\int_{|\xi|\le j}|\xi|^{4-d}\,\d\xi\\
&\ge\frac{c}{[\log(2+j)]^2}\,j^4.
\end{align*}

It follows that
$$
I\ge c\e^{-1}\sum_{j=j_0}^N \frac{1}{(j+1)\log(2+j)},
$$
as since the sum is unbounded as $N\to\infty$ it follows that $u\notin L^1(0,T;\dot H^2)$ as claimed.
\end{proof}

\subsection{Maximal regularity-type results}\label{sec:cheapmaxreg}

Usually, `maximal regularity' results for the heat equation yield
$$
\partial_tu,\quad\Delta u\in L^p(0,T;L^q)
$$
when $u$ solves
$$
\partial_tu-\Delta u=f,\qquad u(0)=0
$$
with $f\in L^p(0,T;L^q)$. The results follows from maximal regularity when $p=q$, obtained as inequality (3.1) in Chapter IV, Section 3, pages 289–290 of \cite{book:LSU} using the Mihlin Multiplier Theorem \citeb{art:Mihlin1957} and then an interpolation theorem due to \cite{art:Benedek1962}; the procedure is clearly explained in \cite{art:Krylov2001}, for example. However, in this paper we will only require the following $L^2$-based Sobolev-space result, for which the basic $L^2(0,T;L^2)$ maximal regularity estimate can be obtained relatively easily.

\begin{proposition}\label{maxreg+}
  There exists a constant $C_r$ such that if $f\in L^r(0,T;H^s)$, $s\ge0$, and $u$ satisfies
  $$
  \partial_tu-\Delta u=f,\qquad u(0)=0,
  $$
  then $u\in L^r(0,T;\dot H^{s+2})$ with
  \begin{equation}\label{eq:SMR}
  \|u\|_{L^r(0,T;\dot{H}^{s+2})}\le C_r\|f\|_{L^r(0,T;H^s)}.
  \end{equation}
  The constant $C_r$ can be chosen uniformly for all $0\le T\le 1$.
\end{proposition}

\begin{proof}
First we treat the case $s=0$, i.e.\ we bound $u\in L^r(0,T;H^2)$ in terms of $f\in L^r(0,T;L^2)$. The passage from estimates for the case $r=2$ to the case $1<r<\infty$ is covered by \cite{art:Krylov2001}. We therefore only prove the estimates in the case $r=2$.

The $L^2$ norm of $u$ is bounded simply by taking the $L^2$ inner product with $u$, using the Cauchy-Schwarz inequality, and integrating in time,
$$
\frac{1}{2}\frac{\d}{\d t}\|u\|^2+\|\nabla u\|^2=\<f,u\>\le\|f\|\|u\|\le\frac{1}{2}\|u\|^2+\frac{1}{2}\|f\|^2
$$
and so
$$
\frac{\d}{\d t}[\e^{-t}\|u\|^2]\le \e^{-t}\|f(t)\|^2,
$$
which implies (since $u(0)=0$) that
$$
\|u(t)\|^2\le\int_0^t\|f(s)\|^2\e^{t-s}\,\d s,
$$
whence
\begin{align*}
\int_0^T\|u(t)\|^2&\le\int_0^T\int_0^t\|f(s)\|^2\e^{t-s}\,\d s\,\d t\\
&=\int_0^T\int_s^T \|f(s)\|^2\e^{t-s}\,\d t\,\d s\le\e^T\int_0^T\|f(s)\|^2\,\d s,
\end{align*}
i.e.
\begin{equation}\label{L2L2L2L2}
\|u\|_{L^2(0,T;L^2)}\le\e^T\|f\|_{L^2(0,T;L^2)}.
\end{equation}

For $\|u\|_{\dot H^2}$ we can argue directly from the Fourier transform of the solution $u$, since
  $$
  \widehat{\Lambda^2u}(\xi,t)=\int_0^t |\xi|^2\e^{-|\xi|^2(t-s)}\hat f(s,\xi)\,\d s.
  $$
  Define
  $$
  G(\xi,t)=\begin{cases} |\xi|^2\e^{-|\xi|^2t}&t\ge0\\
  0&t<0
  \end{cases}
  $$
  and
  $$
  F(\xi,t)=\begin{cases}\hat f(\xi;t)&0\le t\le T\\
  0&\mbox{otherwise}.
  \end{cases}
  $$
  Then
  $$
  \widehat{\Lambda^2u}(\xi;t)=\int_\R G(\xi,t-s)F(\xi,s)\,\d s
  $$
and we can use Young's inequality for convolutions to give
  $$
\|\widehat{\Lambda^2u}(\xi,\cdot)\|_{L^2(0,T)}\le\|\widehat{\Lambda^2u}(\xi,\cdot)\|_{L^2(\R)}\le\|F(\xi,\cdot)\|_{L^2(\R)}=\|\hat f(\xi,\cdot)\|_{L^2(0,T)},
  $$
  since $\|G(\xi,\cdot)\|_{L^1}=1$. Therefore
  $$
  \|u\|_{L^2(0,T;\dot H^2)}=\|\widehat{\Lambda^2u}\|_{L^2((0,T)\times\R^d)}\le\|\hat f\|_{L^2(0,T)\times\R^d}=\|f\|_{L^2(0,T;L^2)}.
  $$

 Combined with (\ref{L2L2L2L2}) this yields $\|u\|_{L^2(0,T;\dot H^2)}\le C_2\|f\|_{L^2(0,T;L^2)}$,
  and hence (via the results of \cite{art:Benedek1962}) we obtain
  $$
  \|u\|_{L^r(0,T;\dot H^2)}\le C_r\|f\|_{L^r(0,T;L^2)}.
  $$
    Letting $C_r$ be the constant for the choice $T=1$, it can easily be seen that this constant is also valid for $T\le 1$ by extending $f$ to be zero on the interval $(T,1]$.

    We now apply this estimate to $u=v$ and $u=\Lambda^sv$: we obtain
    $$
    \|v\|_{L^r(0,T;\dot H^2)}\le C_r\|f\|_{L^r(0,T;L^2)}
    $$
    and
    $$
    \|v\|_{L^r(0,T;\dot H^{s+2})}=\|\Lambda^sv\|_{L^r(0,T;\dot H^2)}\le C_r\|\Lambda^sf\|_{L^r(0,T;L^2)}=C_r\|f\|_{L^r(0,T;\dot H^s)},
    $$
    from which (\ref{eq:SMR}) follows.
\end{proof}

%

%

We will apply this result in combination with the regularity results for the heat equation from Lemma~\ref{lem:good} in the following form for solutions of the Stokes equation, allowing for non-zero initial data. Note that in order to obtain an $L^1$-in-time estimate on $\|u\|_{H^{s+1}}$ we require $L^r$ integrability of $f$ with $r>1$, and the initial data to be in $H^{s-1+\eps}$. Considering the equations in Besov spaces as in \cite{art:BesovMHD} allows $r=1$ and $\eps=0$ (and $s=d/2$ rather than $s>d/2$ in the final results).

\begin{corollary}\label{witheps}
  If $f\in L^r(0,T;H^{s-1})$, $1<r<\infty$, $s>1$, and
  $$
  \partial_tu-\Delta u+\nabla p=f,\qquad\nabla\cdot u=0,\qquad u(0)=u_0\in H^{s-1+\eps},
  $$
  where $u_0$ is divergence free,
  then for $T\le 1$
  \begin{equation}\label{whatweneed}
 \int_0^T \|u\|_{H^{s+1}}\le C_\eps T^{\eps/2}\|u_0\|_{H^{s-1+\eps}}+C_rT^{1-\frac{1}{r}}\|f\|_{L^r(0,T;H^{s-1})}.
  \end{equation}
\end{corollary}

\begin{proof}
  First we consider the solution $v$ of
  $$
  \partial_tv-\Delta v+\nabla\pi=0,\qquad\nabla\cdot v=0,\qquad v(0)=u_0.
  $$
  Since $u_0$ is divergence free, if we apply the Leray projector $\mathbb P$ (orthogonal projection onto elements of $L^2$ whose weak divergence is zero) we obtain
  $$
  \partial_tv-\Delta v=0,\qquad v(0)=u_0,
  $$
  since $\mathbb P$ commutes with derivatives on the whole space. It follows that $v$ is in fact the solution of the heat equation with initial data $u_0$. We can therefore use Lemma~\ref{lem:good} to ensure that
 $$
v\in L^\infty(0,T;H^{s-1+\eps})\cap L^2(0,T;H^{s+\eps})\cap L^q(0,T;H^{s+\eps+1})
$$
for any $q<1$, with all these norms depending only on the norm of the initial data in $H^{s-1+\eps}$. It follows by interpolation that $v\in L^1(0,T;H^{s+1})$ with
\begin{align*}
\int_0^T\|v\|_{H^{s+1}}&\le\int_0^T\|v\|_{H^{s+1+\eps}}^{1-\eps}\|v\|_{H^{s+\eps}}^\eps\\
&\le\left(\int_0^T\|v\|_{H^{s+1+\eps}}^{2(1-\eps)/(2-\eps)}\right)^{(2-\eps)/2}
\left(\int_0^T\|v\|_{H^{s+\eps}}^2\right)^{\eps/2},
\end{align*}
using H\"older's inequality with exponents $(2/(2-\eps),2/\eps)$. Noting that
$$
\frac{2(1-\eps)}{2-\eps}=1-\frac{\eps}{2-\eps}<1
$$
we can use Lemma~\ref{lem:good} to obtain
\begin{equation}\label{MR1}
\int_0^T\|v\|_{H^{s+1}}\le C_\eps T^{\eps/2}\|u_0\|_{H^{s-1+\eps}}.
\end{equation}

The difference $w=u-v$ satisfies
$$
\partial_tw-\Delta w+\nabla\theta=f,\qquad\nabla\cdot w=0,\qquad w(0)=0.
$$
Again we can apply the Leray projector $\mathbb P$, which is bounded from $H^{s-1}$ into $H^{s-1}$, to obtain
$$
\partial_tw-\Delta w={\mathbb P}f,\qquad w(0)=0.
$$
By the maximal regularity results for the heat equation from Proposition~\ref{maxreg+}, we know that for any $r>1$ we have
\begin{align}
\int_0^T\|w\|_{H^{s+1}}&\le T^{1/r'}\|w\|_{L^r(0,T;H^{s+1})}\nonumber\\
&\le C_rT^{1/r'}\|{\mathbb P}f\|_{L^r(0,T;H^{s-1})}\nonumber\\
&\le C_rT^{1/r'}\|f\|_{L^r(0,T;H^{s-1})},\label{squiggle2}
\end{align}
where $(r,r')$ are conjugate. The inequality in (\ref{whatweneed}) now follows by combining (\ref{MR1}) and (\ref{squiggle2}).
  \end{proof}

\section{Proof of local existence}

The main part of the proof consists of a priori estimates, which we prove formally. To make the proof rigorous requires some approximation procedure such as that employed in \cite{art:Commutators}, to which we refer for the details. Where the limiting process involved would turns equalities into inequalities, we write inequalities even in these formal estimates.

\begin{theorem}\label{NLE}
  Let $d = 2,3$. Take $s>d/2$ and $0 < \eps < 1$. Suppose that $B_0\in H^s(\R^d)$ and $u_0\in H^{s-1+\eps}(\R^d)$. Then there exists $T_*>0$ such that the non-resistive MHD system (\ref{MHD}) has a solution $(u,B)$ with
  $$
  u\in L^\infty(0,T_*;H^{s-1+\eps}(\R^d))\cap L^2(0,T_*;H^{s+\eps}(\R^d))\cap L^1(0,T_*;H^{s+1}(\R^d))
  $$
  and
  $$
  B\in L^\infty(0,T_*;H^s(\R^d)).
  $$
\end{theorem}

Note that the case $\eps = 1$ was covered in a previous paper \citep{art:Commutators}, and is specifically excluded here.

\begin{proof} Throughout the proof various constants will depend on $s$, but we do not track this dependency.

We first obtain a basic energy estimate in $L^2$. We take the $L^2$ inner product of the $u$ equation with $u$ and of the $B$ equation with $B$ to obtain
\begin{align*}
\frac{1}{2}\frac{\d}{\d t}\|u\|^2+\|\nabla u\|^2&=\<(B\cdot\nabla)B,u\>\qquad\mbox{and}\\
\frac{1}{2}\frac{\d}{\d t}\|B\|^2&=\<(B\cdot\nabla)u,B\>.
\end{align*}
Since  $\<(B\cdot\nabla)u,B\>=-\<(B\cdot\nabla)u,B\>$ we can add the two equations to yield
$$
\frac{1}{2}\frac{\d}{\d t}\left(\|u\|^2+\|B\|^2\right)+\|\nabla u\|^2\le0
$$
and so
\begin{equation}\label{BEE}
\|u(t)\|^2+\|B(t)\|^2+2\int_0^t\|\nabla u(s)\|^2\,\d s\le \|u_0\|^2+\|B_0\|^2=:M_0.
\end{equation}

It is also helpful to have two other estimates for later use; observing that
$$
|\<(B\cdot\nabla)u,B\>|\le c\|B\|\|\nabla u\|\|B\|_{L^\infty}
$$
and using the embedding $H^s\subset L^\infty$ (valid since $s>d/2$) and Young's inequality we obtain
\begin{equation}\label{upoor}
\frac{\d}{\d t}\|u\|^2+\|\nabla u\|^2\le c\|B\|^2\|B\|_{H^s}^2\le c\|B\|_{H^s}^4;
\end{equation}
and similarly, since $|\<(B\cdot\nabla)u,B\>|\le\|B\|^2\|\nabla u\|_{L^\infty}$,
\begin{equation}\label{Bpoor}
\frac{1}{2}\frac{\d}{\d t}\|B\|^2\le c\|\nabla u\|_{H^s}\|B\|^2.
\end{equation}

In order to estimate the norm of $B$ in $H^s$ we act on the $B$ equation with $\Lambda^s$ and take the inner product with $\Lambda^sB$ in $L^2$. This yields
\begin{align}
\frac{1}{2}\frac{\d}{\d t}\|\Lambda^sB\|^2&\le\left|\<\Lambda^s[(B\cdot\nabla)u],\Lambda^sB\>\right|+\left|\<\Lambda^s[(u\cdot\nabla)B],\Lambda^sB\>\right|\nonumber\\
&\le c\|\nabla u\|_{H^s}\|B\|_{H^s}^2\label{Best1}
\end{align}
using the fact that $H^s$ is an algebra since $s>d/2$, along with the estimate proved in \cite{art:Commutators}
$$
\left|\<\Lambda^s[(u\cdot\nabla)B],\Lambda^sB\>\right|\le c\|\nabla u\|_{H^s}\|B\|_{H^s}^2,
$$
valid when $s>d/2$. Combined with (\ref{Bpoor}) this yields
\begin{equation}\label{Bgood}
\frac{1}{2}\frac{\d}{\d t}\|B\|_{H^s}^2\le c\|\nabla u\|_{H^s}\|B\|_{H^s}^2.
\end{equation}

The estimates for the $u$ equation are more delicate. First we obtain estimates on $u$ in the space $L^1(0,T;H^{s+1})$, using the maximal regularity estimates from Proposition~\ref{maxreg+} and Corollary~\ref{witheps}. We consider the equation for $u$ as the forced Stokes equation
$$
\partial_tu-\Delta u+\nabla p=f:=-(u\cdot\nabla)u+(B\cdot\nabla)B,\qquad\nabla\cdot u=0,\qquad u(0)=u_0,
$$
and so estimate (\ref{whatweneed}) from Corollary~\ref{witheps} yields
\begin{equation}\label{squiggle}
 \int_0^T \|u\|_{H^{s+1}}\le C_\eps T^{\eps/2}\|u_0\|_{H^{s-1+\eps}}+C_rT^{1-\frac{1}{r}}\|f\|_{L^r(0,T;H^{s-1})}.
\end{equation}



Since
\begin{align*}
\|f\|_{H^{s-1}}&=\|\nabla\cdot(B\otimes B)-\nabla\cdot(u\otimes u)\|_{H^{s-1}}\\
&\le\|B\otimes B\|_{H^s}+\|u\otimes u\|_{H^s}\\
&\le c\|B\|_{H^s}^2+c\|u\|_{H^s}^2\\
&\le c\|B\|_{H^s}^2+c\|u\|^{2\eps/(s+\eps)}\|u\|_{H^{s+\eps}}^{2s/(s+\eps)}\\
&\le \|B\|_{H^s}^2+cM_0^{\eps/(s+\eps)}\|u\|_{H^{s+\eps}}^{2s/(s+\eps)},
\end{align*}
using (\ref{BEE}), if we choose $r=(s+\eps)/s>1$ then from (\ref{squiggle}) we have
\begin{align}
\int_0^T \|u\|_{H^{s+1}}&\le C_\eps T^{\eps/2}\|u_0\|_{H^{s-1+\eps}}\nonumber\\
&\qquad+C_\eps T^{\eps/(s+\eps)}\left(\int_0^T \|B\|_{H^s}^{2(s+\eps)/s}+cM_0^{\eps/s}\|u\|_{H^{s+\eps}}^{2} \,\d\tau\right)^{s/(s+\eps)}.\label{DMC}
\end{align}


We now estimate the norm of $u$ in $H^{s-1+\eps}$ and $H^{s+\eps}$. If we act with $\Lambda^{s-1+\eps}$ on the $u$ equation and take the inner product with $\Lambda^{s-1+\eps}u$, then
\begin{align}
\frac{1}{2}\frac{\d}{\d t}\|&\Lambda^{s-1+\eps}u\|^2+\|\Lambda^{s+\eps}u\|^2\nonumber\\
&\le -\<\Lambda^{s-1+\eps}[(u\cdot\nabla)u],\Lambda^{s-1+\eps}u\>+\<\Lambda^{s-1+\eps}[(B\cdot\nabla)B],\Lambda^{s-1+\eps}u\>.\label{add1}
\end{align}
For the first term on the right-hand side, we write
\begin{align*}
|\langle \Lambda^{s-1+\eps} [(u \cdot \nabla) u], \Lambda^{s-1+\eps} u \rangle| &=  |\langle \Lambda^{s-1} [(u \cdot \nabla) u], \Lambda^{s-1+2\eps} u \rangle| \\
&=  |\langle \Lambda^{s-1} [\nabla \cdot (u \otimes u)], \Lambda^{s-1+2\eps} u \rangle| \\
&\leq c \|u\|_{H^{s}}^{2} \|u\|_{H^{s-1+2\eps}} \\
&\leq c (\|u\|_{H^{s-1+\eps}}^{\eps} \|u\|_{H^{s+\eps}}^{1-\eps})^{2} \|u\|_{H^{s-1+\eps}}^{1-\eps} \|u\|_{H^{s+\eps}}^{\eps} \\
&\leq c \|u\|_{H^{s-1+\eps}}^{1+\eps} \|u\|_{H^{s+\eps}}^{2-\eps} \\
&\leq c \|u\|_{H^{s-1+\eps}}^{2(1+\eps)/\eps} + \frac{1}{4} \|u\|_{H^{s+\eps}}^{2},
\end{align*}
where we have used Sobolev interpolation, Young's inequality, and the fact that $H^{s}$ is an algebra (as $s > d/2$). The second term is handled similarly:
\begin{align*}
|\langle \Lambda^{s-1+\eps} [(B \cdot \nabla) B], \Lambda^{s-1+\eps} u \rangle| &= |\langle \Lambda^{s-1} [(B \cdot \nabla) B], \Lambda^{s-1+2\eps} u \rangle| \\
&= |\langle \Lambda^{s-1} [\nabla \cdot (B \otimes B)], \Lambda^{s-1+2\eps} u \rangle| \\
&\leq c \|B\|_{H^{s}}^{2} \|u\|_{H^{s-1+2\eps}} \\
&\leq c \|B\|_{H^{s}}^{2} \|u\|_{H^{s-1+\eps}}^{1-\eps} \|u\|_{H^{s+\eps}}^{\eps} \\
&\leq c \|B\|_{H^{s}}^{2(1+\eps)} + c \|u\|_{H^{s-1+\eps}}^{2(1+\eps)/\eps} + \frac{1}{4} \|u\|_{H^{s+\eps}}^{2},
\end{align*}
using the three-term Young's inequality with exponents $(1+\eps, \frac{2(1+\eps)}{\eps(1-\eps)}, \frac{2}{\eps})$.
Combining these yields\footnote{Note that the exponent $2(1+\eps)/\eps$ on the $H^{s-1+\eps}$ norm of $u$ is far from optimal, and can be reduced to some $\gamma$ for $2 < \gamma \leq 4$ by using Lemma 1.1(i) from \cite{art:Chemin1992}. However, the proof here is significantly simpler, and still yields a short-time existence result (albeit with a possibly shorter existence time).}
$$
\frac{1}{2}\frac{\d}{\d t}\|\Lambda^{s-1+\eps}u\|^2+\|\Lambda^{s+\eps}u\|^2\le c \|B\|_{H^{s}}^{2(1+\eps)} + c \|u\|_{H^{s-1+\eps}}^{2(1+\eps)/\eps} + \frac{1}{2} \|u\|_{H^{s+\eps}}^{2}.
$$
If we add (\ref{upoor}) and an additional term $+\|u\|^2$ to both sides then we obtain
\begin{align*}
&\frac{1}{2}\frac{\d}{\d t}\|u\|_{H^{s-1+\eps}}^2 + \|u\|_{H^{s+\eps}}^2 \\
&\qquad \leq c\|u\|_{H^{s-1+\eps}}^{2(1+\eps)/\eps} + c\|B\|_{H^s}^{2(1+\eps)} + c\|B\|_{H^{s}}^{4} + \frac{1}{2}\|u\|_{H^{s+\eps}}^2 + \|u\|^2,
\end{align*}
and so
\begin{equation}\label{eq(3)}
\frac{\d}{\d t}\|u\|_{H^{s-1+\eps}}^2 + \|u\|_{H^{s+\eps}}^2 \leq  c_1 \|u\|_{H^{s-1+\eps}}^{2(1+\eps)/\eps} + c_2 \|B\|_{H^s}^{2(1+\eps)} + c_3 \|B\|_{H^{s}}^{4} + 2\|u\|^2.
\end{equation}

We now have three ingredients: the differential inequality (\ref{eq(3)}) for $u$;
the $B$ equation (\ref{Bgood})
\[
\frac{1}{2}\frac{\d}{\d t}\|B\|_{H^s}^2\le c_4\|\nabla u\|_{H^s}\|B\|_{H^s}^2,
\]
which implies that
\begin{equation}\label{eq(1)}
\|B(t)\|_{H^s}^2\le\|B_0\|_{H^s}^2\exp\left(2c_4\int_0^t\|\nabla u\|_{H^s}\,\d\tau\right);
\end{equation}
and the maximal regularity estimate
\begin{align}
\int_0^T \|u\|_{H^{s+1}}&\le C_\eps T^{\eps/2}\|u_0\|_{H^{s-1+\eps}}\nonumber\\
&\qquad+C_\eps T^{\eps/(s+\eps)}\left(\int_0^T   \|B\|_{H^s}^{2(s+\eps)/s} + c_5 M_0^{\eps/s}\|u\|_{H^{s+\eps}}^{2}  \,\d\tau\right)^{s/(s+\eps)}.\label{eq(2)}
\end{align}

We will now choose $T^*$ such that $\|B(t)\|_{H^s}\le 2\|B_0\|_{H^s}$ for all $t\in[0,T_*]$. Set
\begin{align*}
M_1 &:= \| u_0 \|_{H^{s-1+\eps}}, \\
\text{and} \qquad M_2 &:= 2^{2(1+\eps)}c_2\|B_0\|_{H^s}^{2(1+\eps)} + 2^4 c_3 \|B_0\|_{H^{s}}^4 +2M_0,
\end{align*}
and choose $T^*$ sufficiently small that
\begin{equation}
\label{Tstar1}
0<\left( 1 - \frac{c_1 T (M_1^2 + T M_2)^{1/\eps}}{\eps} \right)^{-\eps} < 2 \qquad \text{for all }0 < T < T_*
\end{equation}
and
\begin{align}
&C_\eps T^{\eps/2} M_1 + C_\eps T^{\eps/(s+\eps)} \Big( 2^{2(s+\eps)/s} T \|B_0\|_{H^s}^{2(s+\eps)/s} \notag \\
& \qquad + c_5 M_0^{\eps/s} \left[ c_1 T [ 2(M_1^2 + T M_2) ]^{(1+\eps)/\eps} + T M_2 \right] \Big)^{s/(s+\eps)}<\frac{\log 4}{2c_4}\label{Tstar2}
\end{align}
for all $0 < T < T_*$.

To show that $\| B(t) \|_{H^{s}} \leq 2 \| B_{0} \|_{H^{s}}$ for $t \in [0, T_*]$, we assume that that $t \mapsto \| B(t) \|_{H^{s}}$ is a continuous function that takes the value $\| B_{0} \|_{H^{s}}$ at time $t=0$. While we have not shown this as part of our formal calculations, it would be true for any member of the family of smooth approximations considered in \cite{art:Commutators}, and the estimates we now obtain would hold uniformly (for this family) for all $t\in[0,T^*]$ for the time $T^*$ defined by (\ref{Tstar1}) and (\ref{Tstar2}).

Set
$$
T=\sup\,\left\{T_0\in[0,T^*] : \|B(t)\|_{H^s}\le 2\|B_0\|_{H^s}\mbox{ for all }t\in[0,T_0]\right\}
$$
and suppose that $T<T^*$. Then from \eqref{BEE}, \eqref{upoor}, and \eqref{eq(3)} we obtain
\begin{align}
&\frac{\d}{\d t}\|u\|_{H^{s-1+\eps}}^2+\|u\|_{H^{s+\eps}}^2 \nonumber \\
&\qquad \leq c_1\|u\|_{H^{s-1+\eps}}^{2(1+\eps)/\eps}+2^{2(1+\eps)}c_2\|B_0\|_{H^s}^{2(1+\eps)} + 2^4 c_3 \|B_0\|_{H^{s}}^4 + 2M_0 \nonumber \\
&\qquad \leq c_1\|u\|_{H^{s-1+\eps}}^{2(1+\eps)/\eps} + M_2 \label{ubound1}
\end{align}
for all $t\in[0,T]$. Using standard ODE comparison techniques (see, for example, Theorem 6 in \cite{art:BNR2015}, where we take $p = (1+\eps)/\eps$) we obtain the bound
\begin{align}
\| u(t) \|_{H^{s-1+\eps}}^{2} &\leq (M_1^2 + T M_2) \left( 1 - \frac{c_1 T (M_1^2 + T M_2)^{1/\eps}}{\eps} \right)^{-\eps}\nonumber\\
&\le 2(M_1^2 + T M_2)\label{ubound3}
\end{align}
for all $t\in[0,T]$, by (\ref{Tstar1}).

Now, substituting \eqref{ubound3} into \eqref{ubound1} and integrating between times $0$ and $T$ yields
\begin{equation}
\label{ubound4}
\int_{0}^{T} \| u(t) \|_{H^{s+\eps}}^2 \, \d t \leq c_1 T [ 2(M_1^2 + T M_2) ]^{(1+\eps)/\eps} + T M_2.
\end{equation}
Substituting \eqref{ubound4} and $\| B(t) \|_{H^{s}} \leq 2 \| B_{0} \|_{H^{s}}$ into \eqref{eq(2)}, we obtain
\begin{align}
\int_0^T \|u(t)\|_{H^{s+1}} \, \d \tau &\leq C_\eps T^{\eps/2} M_1 + C_\eps T^{\eps/(s+\eps)} \Big( 2^{2(s+\eps)/s} T \|B_0\|_{H^s}^{2(s+\eps)/s} \notag \\
& \qquad + c_5 M_0^{\eps/s} \left[ c_1 T [ 2(M_1^2 + T M_2) ]^{(1+\eps)/\eps} + T M_2 \right] \Big)^{s/(s+\eps)}\notag\\
&< \frac{\log 4}{2c_4},\label{ubound5}
\end{align}
using (\ref{Tstar2}).

Substituting this into \eqref{eq(1)} ensures that $\| B(t) \|_{H^{s}} < 2 \| B_{0} \|_{H^{s}}$ for all $t\in[0,T]$, contradicting the maximality of $T$. It follows that $T=T^*$ and hence
$$
\|B(t)\|_{H^s}\le 2\|B_0\|_{H^s}\qquad\mbox{for all}\qquad t\in[0,T^*].
$$
The result now follows from \eqref{ubound3}, \eqref{ubound4}, and \eqref{ubound5}.
%
\end{proof}

\section*{Conclusion}

In the scale of Sobolev spaces we suspect that the result that we have proved here is optimal. \cite{art:BourgainLi2015} showed that the Euler equations on $\R^d$ are ill posed in $H^{1+d/2}$  for $n=2,3$, and we have shown via an explicit example that for the heat equation we cannot gain the time integrability of two additional derivatives that is required in our local existence argument. It would be interesting to find a simpler model problem in which it is possible to demonstrate the failure of local existence for $B_0\in H^s$ and $u_0\in H^{s-1}$.

\section*{Conflict of Interest}

The authors declare that they have no conflict of interest.

\addcontentsline{toc}{section}{Bibliography}

\bibliographystyle{agsm}
\bibliography{MHDssmo}

\end{document}